\documentclass{article}

\usepackage{a4wide}
\usepackage{amsmath}
\usepackage{amsthm}
\usepackage{amssymb}
\usepackage{xcolor}

\usepackage[colorlinks=true,urlcolor=blue,linkcolor=blue,plainpages=false,pdfpagelabels
]{hyperref}

\newtheorem{theorem}{Theorem}[section]
\newtheorem{proposition}[theorem]{Proposition}
\newtheorem{corollary}[theorem]{Corollary}
\newtheorem{lemma}[theorem]{Lemma}

\newcommand{\N}{{\mathbb N}}
\newcommand{\R}{{\mathbb R}}

\newcommand{\Mrn}{M_r}
\newcommand{\Mabn}{M_{\alpha,\beta}}
\newcommand{\Mxnz}{M_0}
\newcommand{\bxz}{b}
\newcommand{\Mxn}{M}
\newcommand{\defP}{\Omega}
\newcommand{\deftP}{\tilde{\defP}}

\title{Rates of ($T$-)asymptotic regularity of the generalized Krasnoselskii-Mann-type  iteration}
\author{Paulo Firmino${}^{a}$ and Lauren{\c t}iu Leu{\c s}tean${}^{b,c,d}$\\[2mm]
\footnotesize ${}^{a}$ Departamento de Matem\'atica, Faculdade de Ci\^encias, Universidade de Lisboa\\ 
\footnotesize ${}^{b}$ LOS, Faculty of Mathematics and Computer Science, University of Bucharest\\
\footnotesize ${}^{c}$ Simion Stoilow Institute of Mathematics of the Romanian Academy\\
\footnotesize ${}^{d}$ Institute for Logic and Data Science, Bucharest\\[1mm]
\footnotesize Emails: \protect\url{fc49883@alunos.ciencias.ulisboa.pt }, \protect\url{laurentiu.leustean@unibuc.ro}
}

\begin{document}

\date{}

\maketitle

\begin{abstract}

\noindent
In this paper we use proof mining methods to compute rates of  ($T$-)asymptotic regularity of the generalized Krasnoselskii-Mann-type iteration associated to a nonexpansive mapping 
$T:X\to X$ in a uniformly convex normed space $X$.  For special choices of the parameter sequences, we obtain quadratic rates.\\



\end{abstract}

\section{Introduction}

Let $X$ be a normed space and $T:X\to X$ be a nonexpansive mapping with fixed points. 
The \emph{generalized Krasnoselskii-Mann-type iteration}, studied by Kanzow and Shehu \cite{KanShe17} in Hilbert spaces and by 
Zhang, Guo, Wang \cite{ZhaGuoWan21} in classes of uniformly convex Banach spaces, is defined as follows:
\begin{equation}\label{def-gen-KM}
x_0=x, \quad x_{n+1}=\alpha_nx_n+\beta_nTx_n + r_n,
\end{equation}
where $(r_n)_{n \in \N} \subseteq X$ and $(\alpha_n)_{n \in \N}, \,(\beta_n)_{n \in \N}  \subseteq [0,1]$ satisfy  
$\alpha_n+\beta_n \le 1$ for all $n\in\N$.

Obviously, if $r_n=0$ and $\alpha_n+\beta_n = 1$ for all $n\in\N$, 
then $(x_n)$ becomes the well-known Krasnoselskii-Mann iteration \cite{Kra55,Man53,Gro72}. As pointed out in \cite{KanShe17,ZhaGuoWan21}, 
by letting $\alpha_n=1-\beta_n$ and $r_n=\beta_ne_n$, where $(e_n) \subseteq X$ is a sequence of error terms, one obtains the inexact Krasnoselskii–Mann iteration studied by 
Combettes \cite{Com04} in Hilbert spaces, by Kim and Xu \cite{KimXu07} in Banach spaces and, more recently, by Liang, Fadili and Peyr\'e \cite{LiaFadPey16}.
Furthermore, by taking $r_n=\delta_n u$, where $u\in X$ and $(\delta_n) \subseteq [0,1]$ is a sequence satisfying $\alpha_n+\beta_n+\delta_n=1$ for all $n\in\N$, 
we get an iteration studied by Yao, Liou, Zhou \cite{YaoLioZho09}, Hu \cite{Hu08} and Hu and Liu \cite{HuLiu09}.

Kanzow and Shehu \cite{KanShe17} prove, in Hilbert spaces,  the weak convergence of the  generalized Krasnoselskii-Mann-type iteration $(x_n)$ to a fixed point of $T$ under 
some  hypotheses on the parameter sequences $(\alpha_n)$, $(\beta_n)$, $(r_n)$. Zhang, Guo, Wang \cite{ZhaGuoWan21} generalize this weak 
convergence result to uniformly convex Banach spaces that satisfy additional properties. By analyzing the proofs from \cite{ZhaGuoWan21} one 
can see that, as it happens for numerous weak/strong convergence proofs of nonlinear iterations, an essential intermediate step is to obtain 
the ($T$-)asymptotic regularity \cite{BroPet66,BorReiSha92} of $(x_n)$.

We say that $(x_n)$ is $T$-asymptotically regular if $\lim\limits_{n\to \infty} \|x_n-Tx_n\|=0$ and that 
$(x_n)$ is asymptotically regular if $\lim\limits_{n\to \infty} \|x_{n+1}-x_n\|=0$.  A mapping 
$\Phi:\N \to \N$ is said to be 
\begin{enumerate}
\item  a rate of $T$-asymptotic regularity of $(x_n)$ if $\Phi$ is a rate of convergence towards $0$ of the sequence $(\|x_n-Tx_n\|)$;
\item  a rate of asymptotic regularity of $(x_n)$ if $\Phi$ is a rate of convergence towards $0$ of the sequence $(\|x_{n+1}-x_n\|)$.
\end{enumerate}
If $\Phi$ is a rate of ($T$-)asymptotic regularity of $(x_n)$,  we also say that $(x_n)$ is ($T$-)asymptotically regular with rate $\Phi$.

In this paper we apply methods of proof mining \cite{Koh08,Koh19} developed by the second author for the 
Ishikawa iteration of a nonexpansive mapping \cite{Leu10,Leu14}
to  compute for the first time uniform rates of ($T$-)asymptotic regularity of the generalized 
Krasnoselskii-Mann-type iteration $(x_n)$ in 
uniformly convex normed spaces. Furthermore, quadratic 
such rates are obtained for special choices of the parameter sequences.

\mbox{}

Notations: $\N^*=\N\setminus \{0\}$, \, $\R_+=[0,\infty)$.  If $m,n\in\N$, $n\geq m$, we write $[m,n]=\{m,m+1,\ldots, n\}$.

\section{Generalized Krasnoselskii-Mann-type iteration}

In the following, $X$ is a normed space, $T:X\to X$ a nonexpansive mapping and  $(x_n)$ is the 
generalized Krasnoselskii-Mann-type iteration defined by \eqref{def-gen-KM}. We denote by 
$Fix(T)$ the set of fixed points of $T$.

For every $z\in X$, let $(K_{z,n})$ be a sequence of real numbers defined as follows:
\begin{align*}
K_{z,0} & =\|x-z\|, \\
K_{z,n} & =\|x-z\|+\sum_{i=0}^{n-1}\beta_i \|Tz-z\| + 
\sum_{i=0}^{n-1}(1-\alpha_i-\beta_i)\|z\| + \sum_{i=0}^{n-1}\|r_i\| \quad\text{~~for $n\geq 1$}.
\end{align*}

Thus, for all $n\geq 0$,
\begin{align*}
K_{z,n+1} &= K_{z,n}+\beta_n\|Tz-z\|+ (1-\alpha_n-\beta_n)\|z\|+\|r_n\|.
\end{align*}

\begin{lemma}
For all $z\in X$ and all $n\in\N$, 
\begin{align}
\|x_{n+1}-z\| & \leq (\alpha_n+\beta_n)\|x_n-z\| +\beta_n\|Tz-z\|+ (1-\alpha_n-\beta_n)\|z\|+\|r_n\|,
\label{ineq-xn+1-z-xn-z} \\
\|x_n-z\| & \leq K_{z,n}, \label{ineq-xn-z-Kzn}\\
\|x_{n+1}-x_n\| & \leq 2K_{z,n+1}, \label{ineq-xn+1-xn}\\
\|x_n-Tx_n\| & \le 2\|x_n-z\|+\|z-Tz\|\leq 2K_{z,n} + \|z-Tz\|. \label{ineq-xn-Txn}
\end{align}
\end{lemma}
\begin{proof}
We have that for all $n\in\N$,
\begin{align*}
\|x_{n+1}-z\| &=\|\alpha_nx_n+\beta_nTx_n+r_n-z\|= \|\alpha_n(x_n-z)+\beta_n(Tx_n-z)+
r_n-(1-\alpha_n-\beta_n)z\|\\
&\le \alpha_n\|x_n-z\|+\beta_n\|Tx_n-z\|+\|r_n-(1-\alpha_n-\beta_n)z\|\\
&\le \alpha_n\|x_n-z\|+\beta_n\|Tx_n-Tz\|+\beta_n\|Tz-z\|+\|r_n\|+\|(1-\alpha_n-\beta_n)z\|\\
&\le (\alpha_n+\beta_n)\|x_n-z\| +\beta_n\|Tz-z\|+ (1-\alpha_n-\beta_n)\|z\|+\|r_n\|,
\end{align*}
as $T$ is nonexpansive and $\alpha_n+\beta_n\leq 1$.
Thus, \eqref{ineq-xn+1-z-xn-z} holds. 

\eqref{ineq-xn-z-Kzn} is proved by induction on $n$. The case $n=0$ is obvious.

$n \Rightarrow n+1$:
\begin{align*}
\|x_{n+1}-z\| & \stackrel{\eqref{ineq-xn+1-z-xn-z}}{\le} (\alpha_n+\beta_n)
\|x_n-z\|+\beta_n\|Tz-z\|+ (1-\alpha_n-\beta_n)\|z\| + 
\|r_n\|\\
& \leq \|x_n-z\|+\beta_n\|Tz-z\|+ (1-\alpha_n-\beta_n)\|z\| + 
\|r_n\| \quad \text{as~}\alpha_n+\beta_n \le 1\\
& \leq K_{z,n} + \beta_n\|Tz-z\|+(1-\alpha_n-\beta_n)\|z\| + 
\|r_n\| \text{~by the induction hypothesis}\\
& = K_{z,n+1}.
\end{align*}

\eqref{ineq-xn+1-xn} and \eqref{ineq-xn-Txn} follow easily: 
\begin{align*}
\|x_{n+1}-x_n\| &\leq \|x_{n+1}-z\| + \|x_n-z\| \stackrel{\eqref{ineq-xn-z-Kzn}}{\le} K_{z,n+1} + K_{z,n} 
\leq 2K_{z,n+1},\\
\|x_n-Tx_n\| &\le \|x_n-z\|+\|z-Tz\|+\|Tz-Tx_n\|\le 2\|x_n-z\|+\|z-Tz\|\\
& \stackrel{\eqref{ineq-xn-z-Kzn}}{\le} 2K_{z,n} + \|z-Tz\|.
\end{align*}
\end{proof}

A very useful immediate consequence of \eqref{ineq-xn-Txn} is the following. 

\begin{corollary}\label{cor-xn-Txn-xn-z-fixed-point}
Assume that $z\in Fix(T)$. Then for all $n\in\N$, 
$$
\|x_n-Tx_n\|  \leq 2\|x_n-z\|\leq  2K_{z,n}. 
$$
\end{corollary}

\begin{lemma}
For all $n\in \N$, 
\begin{align}
\|x_{n+1}-x_n\| & \le \beta_n\|x_n-Tx_n\|+(1-\alpha_n-\beta_n)\|x_n\| + \|r_n\|, \label{ineq-xnp1-xn}\\
\|x_{n+1}-Tx_{n+1}\| & \le \|x_n-Tx_n\| + 2(1-\alpha_n-\beta_n)\|x_n\| +  2 \|r_n\|.  \label{ineq-xnp1-Txnp1}
\end{align}
\end{lemma}
\begin{proof}
Let $n\in \N$. We get that
\begin{align*}
\|x_{n+1}-x_n\| &=  \|\beta_n(Tx_n-x_n)-(1-\alpha_n-\beta_n)x_n + r_n\|\\
&\le  \beta_n\|x_n-Tx_n\|+(1-\alpha_n-\beta_n)\|x_n\| + \|r_n\|
\end{align*}
and 
\begin{align*}
\|Tx_{n+1}-x_{n+1}\| &= \|Tx_{n+1}-Tx_n+(1-\beta_n)(Tx_n-x_n)+(1-\alpha_n-\beta_n)x_n - r_n\|\\
&\le \|Tx_{n+1}-Tx_n\|+(1-\beta_n)\|Tx_n-x_n\|+ (1-\alpha_n-\beta_n)\|x_n\| + \|r_n\| \\
&\le \|x_{n+1}-x_n\|+(1-\beta_n)\|Tx_n-x_n\|+ (1-\alpha_n-\beta_n)\|x_n\| + \|r_n\| \\
&\stackrel{\eqref{ineq-xnp1-xn}}{\le}  \|x_n-Tx_n\| + 2(1-\alpha_n-\beta_n)\|x_n\| + 2\|r_n\|.
\end{align*}
\end{proof}

\subsection{A useful lemma in uniformly convex normed spaces}

Let us recall that a normed space $X$ is uniformly convex if for all $\varepsilon \in (0,2]$ there 
exists $\delta\in(0,1]$  such that for  all  $x,y\in X$, 
\begin{equation}
\| x\| \le 1, \quad \|y\|\le 1 \text{ and }  \|x-y\|\geq \varepsilon  \text{ imply } 
\left\|\frac12(x+y)\right\|\leq  1-\delta. \label{uc-normed-def}
\end{equation}
A modulus of uniform convexity of $X$ is a mapping $\eta:(0,2]\to (0,1]$ providing a 
$\delta:=\eta(\varepsilon)$ satisfying \eqref{uc-normed-def} for given $\varepsilon\in(0,2]$. 
Thus, $X$ is uniformly convex if and only if $X$ has a modulus of uniform convexity $\eta$. 

As pointed out in \cite[p. 3457]{KohLeu12} (see also \cite[p. 63]{LinTza79}), the uniformly convex 
Banach spaces $L_p$ ($1<p<\infty$) have a modulus of uniform convexity $\eta_p$ given by
\begin{align*}
\eta_p(\varepsilon)  & = \begin{cases} 
\displaystyle \frac{(p-1)\varepsilon^2}8 &  \text{if $1<p<2$},\\[2mm]
\displaystyle \frac{\varepsilon^p}{p\cdot 2^p} & \text{if $2\le p <\infty$}. 
\end{cases} 
\end{align*}
Thus, $\displaystyle \eta(\varepsilon)=\frac{\varepsilon^2}{8}$ is a modulus of uniform convexity 
for a Hilbert space. \\

In the sequel, we assume that $X$ is uniformly convex and $\eta$ is a modulus of uniform convexity 
of $X$. 

\begin{lemma}\label{uc-normed-def-modulus-r-lambda}
Let $\varepsilon \in (0,2]$, $r>0$ and $a,x,y\in X$ be such that 
$$\| x-a\| \le r,\, \|y-a\|\le r \ \text{and} \ \|x-y\|\geq \varepsilon r.$$
Then for all $\lambda \in [0,1]$,
$$
\|(1-\lambda)x + \lambda y- a\|  \le (1-2\lambda(1-\lambda)\eta(\varepsilon))r. 
$$
\end{lemma}
\begin{proof} 
Apply \cite[Lemma 3.3]{Koh03} with $x:=\frac1r(y-a)$ and $y:=\frac1r(x-a)$. 
\end{proof}

The following lemma is the main tool in proving the results from Subsection~\ref{section-modulus-liminf}. 

\begin{lemma}\label{main-lemma-xn-uc}
Let $z$ be a fixed point of $T$ and $n\in\N$ be such that  $\alpha_n+\beta_n >0$.
\begin{enumerate}
\item\label{main-lemma-xn-uc-basic} If $\gamma, \delta, \theta>0$ satisfy
\begin{center}
 $\gamma \leq \|x_n-z\|\leq \delta$ and $\|x_n-Tx_n\|\geq \theta$, 
\end{center}
then 
$$ \|\alpha_nx_n+\beta_nTx_n-(\alpha_n+\beta_n)z\| \leq \|x_n-z\| -
\frac{2\gamma\alpha_n\beta_n}{\alpha_n+\beta_n}\eta\left(\frac{\theta}{\delta}\right).$$
\item\label{main-lemma-xn-uc-special} Assume that $\eta$ can be written as $\eta(\varepsilon)=\varepsilon\cdot\tilde{\eta}(\varepsilon)$ 
with $\tilde{\eta}$ increasing.  If $\delta^*,\theta >0$ are such that 
\begin{center}
 $\|x_n-z\|\leq \delta^*$ and $\|x_n-Tx_n\|\geq \theta$, 
\end{center}
then
$$ \|\alpha_nx_n+\beta_nTx_n-(\alpha_n+\beta_n)z\| \leq \|x_n-z\| -
\frac{2\theta\alpha_n\beta_n}{\alpha_n+\beta_n}\tilde{\eta}\left(\frac{\theta}{\delta^*}\right).$$
\end{enumerate}
\end{lemma}
\begin{proof}
\begin{enumerate}
\item  We have that 
\begin{align*}
\alpha_nx_n+\beta_nTx_n & =\frac{\alpha_n}{\alpha_n+\beta_n}(\alpha_n+\beta_n)x_n + 
\frac{\beta_n}{\alpha_n+\beta_n}(\alpha_n+\beta_n)Tx_n,\\
\|(\alpha_n+\beta_n)x_n-(\alpha_n+\beta_n)z\| & =  (\alpha_n+\beta_n)\|x_n-z\|, \\
\|(\alpha_n+\beta_n)Tx_n-(\alpha_n+\beta_n)z\| & = (\alpha_n+\beta_n)\|Tx_n-z\|
 = (\alpha_n+\beta_n)\|Tx_n-Tz\| \\
& \le (\alpha_n+\beta_n)\|x_n-z\|, \\
\|(\alpha_n+\beta_n)x_n-(\alpha_n+\beta_n)Tx_n\| & =(\alpha_n+\beta_n)\|x_n-Tx_n\| 
\ge (\alpha_n+\beta_n)\theta \\
&  = (\alpha_n+\beta_n)\|x_n-z\|\frac{\theta}{\|x_n-z\|}\\
& \ge (\alpha_n+\beta_n)\|x_n-z\|\frac{\theta}{\delta}.
\end{align*}
Obviously, $\frac{\|x_n-Tx_n\|}{\|x_n-z\|}>0$ and, by Corollary \ref{cor-xn-Txn-xn-z-fixed-point}, $\frac{\|x_n-Tx_n\|}{\|x_n-z\|}\leq 2$.

We can apply Lemma \ref{uc-normed-def-modulus-r-lambda} with 
\begin{eqnarray*}
x:=(\alpha_n+\beta_n)x_n,  & y:=(\alpha_n+\beta_n)Tx_n, & a:=(\alpha_n+\beta_n)z, \\
r:=(\alpha_n+\beta_n)\|x_n-z\|, &  \displaystyle \lambda:=\frac{\beta_n}{\alpha_n+\beta_n},  
& \varepsilon:=\frac{\theta}{\delta}
\end{eqnarray*}
to get that 
\begin{align*}
\|\alpha_nx_n+\beta_nTx_n-(\alpha_n+\beta_n)z\| &\le 
\left(1-2\frac{\alpha_n\beta_n}{(\alpha_n+\beta_n)^2}\eta\left(\frac{\theta}{\delta}\right)\right)
(\alpha_n+\beta_n)\|x_n-z\|\\
& =  (\alpha_n+\beta_n)\|x_n-z\|-\frac{2\alpha_n\beta_n}{\alpha_n+\beta_n}\eta\left(\frac{\theta}{\delta}\right)\|x_n-z\|\\
& \leq \|x_n-z\| -\frac{2\alpha_n\beta_n}{\alpha_n+\beta_n}\eta\left(\frac{\theta}{\delta}\right)\gamma.
\end{align*}
\item Remark first that, by Corollary \ref{cor-xn-Txn-xn-z-fixed-point} and hypothesis, we have that 
$\|x_n-z\| \geq \frac{\|x_n-Tx_n\|}2 \geq \frac{\theta}2>0$.

Apply \eqref{main-lemma-xn-uc-basic} with $\gamma:=\delta:=\|x_n-z\|$ to get that 
\begin{align*}
 \|\alpha_nx_n+\beta_nTx_n-(\alpha_n+\beta_n)z\| & \leq    \|x_n-z\| -
\frac{2\|x_n-z\|\alpha_n\beta_n}{\alpha_n+\beta_n}\eta\left(\frac{\theta}{\|x_n-z\|}\right)\\
& =  \|x_n-z\| -
\frac{2\theta\alpha_n\beta_n}{\alpha_n+\beta_n} \tilde{\eta}\left(\frac{\theta}{\|x_n-z\|}\right)\\
& \leq   \|x_n-z\| -\frac{2\theta\alpha_n\beta_n}{\alpha_n+\beta_n} 
\tilde{\eta}\left(\frac{\theta}{\delta^*}\right),
\end{align*}
as $\frac{\theta}{\delta^*}\leq \frac{\theta}{\|x_n-z\|}$, so 
$\tilde{\eta}\left(\frac{\theta}{\delta^*}\right)\leq \tilde{\eta}\left(\frac{\theta}{\|x_n-z\|}\right)$, since 
$\tilde{\eta}$ is increasing.
\end{enumerate}
\end{proof}

\section{Quantitative notions, lemmas and hypotheses}

Firstly, let us recall the quantitative notions used in the paper. Let $(a_n)_{n \in \N} \subseteq \R$ be a 
sequence  of real numbers. We say that a function $\varphi:\N\to\N$ is 
\begin{enumerate}
\item  a rate of convergence of $(a_n)$ towards $a\in \R$ if for all $k\in\N$,
\[|a_n-a|\leq \frac1{k+1} \quad \text{holds for all~}n\geq \varphi(k).\]
\item a Cauchy modulus of $(a_n)$ if for all $k\in\N$,
\[ |a_{n+p} - a_n|\leq \frac1{k+1} \quad \text{holds for all~} n\geq \varphi(k) \text{~and 
all~}p\in\N.\]
\end{enumerate}
If the series $\sum\limits_{n=0}^\infty a_n$  converges, then a Cauchy modulus of the series is a 
Cauchy modulus of the sequence $\left(\tilde{a}_n=\sum\limits_{i=0}^n a_i\right)$
of partial sums. Furthermore, if $\sum\limits_{n=0}^\infty a_n$ diverges, then  $\theta:\N\to\N$ 
is a rate of divergence of the series if $\sum\limits_{i=0}^{\theta(k)} a_i \geq k$ for all $k\in\N$.

The following useful lemmas are immediate. 

\begin{lemma}\label{lemma-series-conv-upper-bound}
Assume that $(a_n) \subseteq \R_+$ and that $\sum\limits_{n=0}^\infty a_n$ converges with Cauchy modulus 
$\varphi$. Then 
\begin{enumerate} \vspace*{-4mm}
\item\label{lemma-series-upper-bound} $\sum\limits_{n=0}^\infty a_n \leq M$, where $M\in\N^*$ is such that 
$M\geq \sum\limits_{i=0}^{\varphi(0)} a_i + 1$.  
\item\label{lemma-series-conv} $\lim\limits_{n\to\infty} a_n=0$ with rate of 
convergence $\psi(k)=\varphi(k)+1$.
\end{enumerate} 
\end{lemma}
\begin{proof}
See the proof of \cite[Lemma 4.3(ii)]{FirLeu24}.
\end{proof}

\begin{lemma}\label{divergence-rate-geq-k} 
Assume that $a_n \in [0,1)$ for all $n\in\N$ and that $\sum\limits_{n=0}^\infty a_n$ diverges with
rate of divergence $\theta$. Then $\theta(n) \ge n$ for all $n\in\N$. 
\end{lemma}
\begin{proof}
Suppose, by contradiction, that for some $n\in\N$, we have that $\theta(n)<n$, hence $\theta(n)\leq n-1$.
Then $\sum\limits_{i=0}^{\theta(n)} a_i \le \sum\limits_{i=0}^{n-1} a_i < n$, which is
a contradiction.
\end{proof}

\begin{lemma}\label{cauchy-modulus-linear-comb}
Let $(a_n), (b_n) \subseteq \R$,  $s,t\in\N^*$ and $c_n=sa_n+tb_n$ for all $n\in\N$. 
Assume that $\varphi_1$ is a Cauchy modulus of $(a_n)$  and $\varphi_2$ is a Cauchy modulus of $(b_n)$.
Define 
\[ \varphi(k)=\max\{\varphi_1(2s(k+1)-1), \varphi_2(2t(k+1)-1)\}.\]
Then $\varphi$ is a Cauchy modulus of $(c_n)$.
\end{lemma}
\begin{proof}
See the proof of \cite[Lemma 3.2]{FirLeu24}.
\end{proof}

\subsection{A quantitative lemma on sequences of nonnegative reals}

A modulus of liminf \cite{Leu14} for a sequence $(a_n)_{n \in \N} \subseteq \R_+$ is a mapping  
$\delta:\N\times\N\to\N$  satisfying the following: 
\begin{center}
for all $k, L\in\N$ there exists 
$N\in [L, \delta(k,L)]$ such that $a_N < \frac1{k+1}$.
\end{center}
Clearly, $\liminf\limits_{n\to\infty} a_n=0$ if and only if 
$(a_n)$ has a modulus of liminf. 

The following quantitative lemma is one of the main tools in the proof of our main result, 
Theorem \ref{main-thm-T-as-reg}.

\begin{lemma}\label{prop-rate-conv-modulus-liminf}
Assume that $\left(a_n\right)_{n \in \N},\left(b_n\right)_{n \in \N}\subseteq \R_+$ satisfy the following:
\begin{enumerate}
\item  $\liminf\limits_{n\to\infty} a_n=0$ with modulus of liminf $\delta$;
\item  $\sum\limits_{n=0}^\infty b_n$  converges with Cauchy modulus $\psi$;
\item  $ a_{n+1} \leq a_n + b_n$ for all $n\in\N$.
\end{enumerate} 
Then $\lim\limits_{n\to\infty} a_n=0$ with rate of convergence $\varphi$ defined by 
\begin{equation}
\varphi(k) = \delta\left(2k+1, \psi(2k+1)+1\right).
\end{equation}
\end{lemma}
\begin{proof}
Denote, for all $n\in\N$, $\tilde{b}_n=\sum\limits_{i=0}^{n}b_i$.
First, let us remark that we get, by induction on $m$, that for all $n,m\in\N$, 
$$ a_{n+m} \leq a_n + \sum\limits_{i=n}^{n+m-1}b_i.$$
Let $k\in\N$ and let us denote $\psi_k=\psi(2k+1)$. As $\delta$ is a modulus of liminf of 
$(a_n)$, there exists $N\in [\psi_k+1,\varphi(k)]$ 
and $a_N < \frac1{2(k+1)}$.

It follows that for all $n\geq \varphi(k)$,
\begin{align*}
a_n & \leq a_N + \sum\limits_{i=N}^{n-1}b_i 
= a_N + \sum\limits_{i=0}^{n-1}b_i - \sum\limits_{i=0}^{N-1}b_i 
= a_N + \tilde{b}_{n-1} - \tilde{b}_{N-1} \\
& = a_N + \tilde{b}_{\psi_k+ q_1} - \tilde{b}_{\psi_k+q_2}, \text{~where~} 
q_1=n-\psi_k-1, q_2= N-1-\psi_k\\
& < \frac1{2(k+1)}  + \tilde{b}_{\psi_k+ q_1} - \tilde{b}_{\psi_k+q_2}.
\end{align*}

As $n\geq \varphi(k) \geq N \geq \psi_k+1$, we have that $q_1,q_2\in\N$, 
so we can apply the fact that $\psi$ is a Cauchy modulus of $(\tilde{b}_n)$ to get that
\begin{align*}
a_n &  < \frac1{2(k+1)} + \frac1{2(k+1)}=\frac1{k+1}
\end{align*}
\end{proof}

\subsection{Quantitative hypotheses on the parameter sequences}

We consider the following quantitative hypotheses on the parameter sequences $(\alpha_n)$, 
$(\beta_n)$, $(r_n)$:
\begin{align*}
(C1) & \quad \sum\limits_{n=0}^{\infty} (1-\alpha_n-\beta_n) < \infty \text{~with Cauchy modulus~}  
\sigma_1, \\
(C2) & \quad \alpha_n+\beta_n >0 \text{~for all $n\in\N$ and~} \sum\limits_{n=0}^{\infty} \frac{\alpha_n\beta_n}{\alpha_n+\beta_n}=\infty 
\text{~with rate of divergence~} \sigma_2,\\
(C3) & \quad \sum\limits_{n=0}^{\infty} \|r_n\| < \infty \text{~with Cauchy modulus~}  \sigma_3.
\end{align*}

Assume that $(C1)$ and $(C3)$ hold. Then the series $\sum\limits_{n=0}^{\infty} (1-\alpha_n-\beta_n)$, 
$\sum\limits_{n=0}^{\infty} \|r_n\|$ are bounded. 

In the sequel, $\Mabn, \Mrn \in\N$ are such that 
\begin{equation}\label{Mabn-Mrn-upper-bounds-series}
\sum\limits_{n=0}^\infty (1-\alpha_n-\beta_n) \leq \Mabn \quad \text{and} \quad 
\sum\limits_{n=0}^\infty r_n \leq \Mrn.
\end{equation}
If $\alpha_n+\beta_n=1$ for all $n\in\N$, then  $\Mabn=0$. Similarly, if 
$r_n=0$ for all $n\in\N$, then  $\Mrn=0$. Otherwise, by 
Lemma~\ref{lemma-series-conv-upper-bound}, we can take 
\begin{equation}\label{Mabn-Mrn-upper-bounds-series-used}
\Mabn \ge\left\lceil\sum\limits_{n=0}^{\sigma_1(0)}(1-\alpha_i-\beta_i)\right\rceil+1
\quad \text{and} \quad 
\Mrn \ge\left\lceil\sum\limits_{n=0}^{\sigma_3(0)}\|r_n\|\right\rceil+1.
\end{equation}

\begin{lemma}\label{lemma-upper-bound-xnz}
Assume that $(C1)$ and $(C3)$ hold and that $z\in Fix(T)$. Then for all $n\in\N$, 
$$
\|x_n-z\| \leq \|x-z\| +  \Mabn\|z\| + \Mrn. 
$$
\end{lemma}
\begin{proof}
By  the definition of $K_{z,n}$, \eqref{ineq-xn-z-Kzn} and 
\eqref{Mabn-Mrn-upper-bounds-series}.
\end{proof}

\section{Main results}

The main result of the paper  computes uniform rates of ($T$-)asymptotic regularity of the generalized Krasnoselkii-Mann-type iteration. 

\begin{theorem}\label{main-thm-T-as-reg}
Let $X$ be  a uniformly convex normed space, $\eta$ a modulus of uniform convexity of $X$, $T:X\to X$  a nonexpansive 
mapping such that $Fix(T)\neq\emptyset$, and $(x_n)$ be given by 
$$
x_0=x\in X, \quad x_{n+1}=\alpha_nx_n+\beta_nTx_n + r_n,
$$
where  $(r_n) \subseteq X$  and $(\alpha_n),(\beta_n) \subseteq [0,1]$ are such 
that $\alpha_n+\beta_n \le 1$ for all $n\in\N$. \\
Suppose that (C1), (C2), (C3) hold, $\Mabn,\Mrn\in \N$ satisfy \eqref{Mabn-Mrn-upper-bounds-series} and 
$\bxz,\Mxnz,\Mxn\in\N^*$  are such that 
\begin{equation}\label{def-bxz-bz-Mxnz}
\bxz\geq \max\{\|x-z\|, \|z\|\}, \quad \Mxnz = \bxz + \Mabn \bxz + \Mrn, \quad \Mxn = \Mxnz + \bxz
\end{equation}
for some fixed point $z$ of $T$.

Define $\Phi,\Psi: \N\to \N$ by 
\begin{align}
\Phi(k) & = \sigma_2\big(\defP(2k+1)+\max\{\sigma_1(8\Mxn(k+1)-1), \sigma_3(8k+7)\}+1\big) 
\label{def-Phi-rate-T-as-reg},\\
\Psi(k) & =  \Phi(2k+1), \label{def-Psi-rate-as-reg}
\end{align}
where 
\begin{equation}\label{def-Pk}
\defP: \N\to \N, \quad \defP(k)=\left\lceil\frac{ (\Mxnz + \Mabn \bxz + \Mrn + 1)(k+1)}{\eta\left(\frac{1}{\Mxnz(k+1)}\right)}
\right\rceil.
\end{equation}

Then $(x_n)$ is $T$-asymptotically regular with rate $\Phi$ and asymptotically regular with rate  $\Psi$.
\end{theorem}

One can get for $(x_n)$, by using Lemma~\ref{main-lemma-xn-uc}.\eqref{main-lemma-xn-uc-special},
a result similar with the one proved by Kohlenbach \cite[Theorem 3.4]{Koh03} 
for the Krasnoselskii-Mann iteration.

\begin{proposition}\label{main-thm-T-as-reg-special-eta} 
In the hypotheses of Theorem~\ref{main-thm-T-as-reg}, assume moreover that  $\eta$ can be written as $\eta(\varepsilon)=\varepsilon\cdot\tilde{\eta}(\varepsilon)$ 
with $\tilde{\eta}$ increasing. Let $\deftP: \N\to \N$ be defined by 
\begin{equation}\label{def-Ptk-nice-modulus}
\deftP(k)= \left\lceil\frac{ (\Mxnz + \Mabn \bxz  + \Mrn + 1)(k+1)}{2\tilde{\eta}\left(\frac{1}{\Mxnz(k+1)}\right)}
\right\rceil.
\end{equation}
Then $(x_n)$ is $T$-asymptotically regular with rate $\tilde{\Phi}$ and asymptotically regular 
with rate  $\tilde{\Psi}$, where $\tilde{\Phi}$ is obtained from $\Phi$ 
 by taking $\deftP$ instead of $\defP$ in \eqref{def-Phi-rate-T-as-reg} 
and $\tilde{\Psi}(k)=\tilde{\Phi}(2k+1)$.
\end{proposition}

The proofs of Theorem~\ref{main-thm-T-as-reg} and Proposition~\ref{main-thm-T-as-reg-special-eta}
are given in Section~\ref{main-results-proofs}.

In the following we give some consequences of our main results. 

\begin{corollary}\label{cor-Hilbert-spaces}
Assume that $X$ is a Hilbert space. Then Proposition~\ref{main-thm-T-as-reg-special-eta} holds with 
\begin{align*}
\deftP(k) & = 4\Mxnz(\Mxnz + \Mabn\bxz  + \Mrn + 1)(k+1)^2.
\end{align*}
\end{corollary}
\begin{proof}
Apply the fact that  a modulus of uniform convexity of $X$ is $\displaystyle \eta(\varepsilon)=\frac{\varepsilon^2}{8}
=\varepsilon\tilde{\eta}(\varepsilon)$, where 
$\displaystyle \tilde{\eta}(\varepsilon) = \frac{\varepsilon}8$ is increasing. 
\end{proof}

\begin{corollary}\label{main-thm-T-as-reg-cor-rn0}
Assume that $r_n=0$ for all $n\in\N$, hence $x_{n+1}=\alpha_nx_n+\beta_nTx_n$ for all $n\in \N$.
Then Theorem~\ref{main-thm-T-as-reg} holds with $\sigma_3(k)=0$, $\Mrn=0$, 
 $\Mxnz=\bxz + \Mabn \bxz$.
\end{corollary}

\begin{corollary}\label{main-thm-T-as-reg-cor-inexact-KM}
Let $(x_n)$ be the inexact Krasnoselskii-Mann iteration:
$$x_0=x\in X, \quad x_{n+1}=(1-\beta_n)x_n + \beta_nTx_n + r_n.$$
Assume that (C2*) and (C3) hold, where 
\begin{align*}
(C2^*) & \quad \sum\limits_{n=0}^{\infty} \beta_n(1-\beta_n)=\infty 
\text{~with rate of divergence~} \sigma_2.
\end{align*}
Define  $ \displaystyle \defP^*: \N\to \N, \,\,
\defP^*(k)=\left\lceil\frac{ (\bxz + 2\Mrn+1)(k+1)}{\eta\left(\frac{1}{(\bxz+\Mrn)(k+1)}\right)}
\right\rceil.$

Then $(x_n)$ is $T$-asymptotically regular with rate 
$\Phi^*(k)  = \sigma_2\left(\defP^*(2k+1)+\sigma_3(8k+7)+1\right)$ and 
asymptotically regular with rate $\Psi^*(k)  =  \Phi^*(2k+1)$.

Furthermore, if $\eta$ can be written as $\eta(\varepsilon)=\varepsilon\cdot\tilde{\eta}(\varepsilon)$ 
with $\tilde{\eta}$ increasing, then one can define $\defP^*$ by 
$ \displaystyle \defP^*(k)=\left\lceil\frac{ (\bxz + 2\Mrn+1)(k+1)}{2\tilde{\eta}\left(\frac{1}{(\bxz+\Mrn)(k+1)}\right)}\right\rceil.$
In particular, if $X$ is a Hilbert space, then
\begin{align}\label{inexact-KM-Hilbert-defP}
\defP^*(k)=4(\bxz + \Mrn)(\bxz + 2\Mrn + 1)(k+1)^2.
\end{align}
\end{corollary}
\begin{proof}
Apply Theorem~\ref{main-thm-T-as-reg} with $\alpha_n=1-\beta_n$ for all $n\in\N$.
As $1-\alpha_n-\beta_n=0$, (C1) obviously holds with $\sigma_1(k)=0$ and $\Mabn=0$ and (C2) 
becomes (C2*). Then $\Mxnz=\bxz + \Mrn$, $\Mxn=2\bxz + \Mrn$. 
If $\eta$ can be written as $\eta(\varepsilon)=\varepsilon\cdot\tilde{\eta}(\varepsilon)$ 
with $\tilde{\eta}$ increasing, apply Proposition~\ref{main-thm-T-as-reg-special-eta}. 
 For the case when $X$ is a Hilbert space,  apply Corollary~\ref{cor-Hilbert-spaces}.
\end{proof}

By combining the previous two corollaries, we get rates of ($T$-)asymptotic regularity for 
the well-known Krasnoselskii-Mann iteration. The rate of $T$-asymptotic regularity that we obtain
for the Krasnoselskii-Mann iteration is very similar with the one obtained 
by Kohlenbach \cite[Theorem~3.4]{Koh03} by applying proof mining methods to a proof due to 
Groetsch \cite{Gro72}.

\begin{corollary}\label{main-thm-T-as-reg-cor-Mann}
Let $(x_n)$ be the Krasnoselskii-Mann iteration:
$$x_0=x\in X, \quad x_{n+1}=(1-\beta_n)x_n+\beta_nTx_n.$$
Then Corollary ~\ref{main-thm-T-as-reg-cor-inexact-KM} holds with $\sigma_3(k)=0$ and $\Mrn=0$.
\end{corollary}

Applying Theorem~\ref{main-thm-T-as-reg} with $r_n=(1-\alpha_n-\beta_n)u$ we get rates for the
following iteration, studied in \cite{YaoLioZho09,Hu08,HuLiu09}.

\begin{corollary}\label{main-thm-T-as-reg-cor-1-alphan-betan-u}
Let $(x_n)$ be the iteration defined as follows:
$$
x_0=x, \quad x_{n+1}=\alpha_nx_n+\beta_nTx_n+ (1-\alpha_n-\beta_n)u,$$ 
where $u\in X$, $u\ne 0$. 

Then Theorem~\ref{main-thm-T-as-reg} holds with $\Mrn=\Mabn\left\lceil \|u\|\right\rceil$ and 
$\sigma_3(k)=\sigma_1(\left\lceil \|u\|\right\rceil(k+1)-1)$. 
\end{corollary}

\section{Examples}

In the sequel, we compute rates of ($T$-)asymptotic regularity of $(x_n)$ for  particular choices 
of the parameter sequences. Let us prove first the following very useful lemma. 
	
\begin{lemma}\label{Cauchy-modulus-bound-u-i+L2}
Let $t\in \R_+$, $L\in \N^*$ and define
\begin{align*}
\varphi(k)=\left\lceil t\right\rceil(k+1), \quad  \varphi^*(k)=\max\{\left\lceil t\right\rceil(k+1)-L, 0\}, 
\quad M_t=\begin{cases} 0 & \text{~if~} t=0, \\
\left\lceil t\left(\frac1L+\frac1{L^2}\right)\right\rceil & \text{~if~} t\ne 0.
\end{cases}
\end{align*}
Then 
\begin{enumerate}
\item $\varphi$  and $\varphi^*$ are Cauchy moduli of $\sum\limits_{n=0}^{\infty}\frac{t}{(n+L)^2}$. 
\item\label{Cauchy-modulus-bound-u-i+L2-bound} $\sum\limits_{n=0}^{\infty}\frac{t}{(n+L)^2}\leq M_t \leq 2\left\lceil t\right\rceil$.
\end{enumerate}
\end{lemma}
\begin{proof}
\begin{enumerate}
\item As $\varphi\geq \varphi^*$, it is enough to prove that  $\varphi^*$ is a Cauchy modulus of 
$\sum\limits_{n=0}^{\infty}\frac{t}{(n+L)^2}$. 
Let  $k\in\N$ and $n\geq \varphi^*(k)$. It follows that $n+L\geq \left\lceil t\right\rceil(k+1)\geq t(k+1)$. 
Then for all $p\in\N$, 
\begin{align*}
\sum\limits_{i=0}^{n+p}\frac{t}{(i+L)^2} - \sum\limits_{i=0}^{n}\frac{t}{(i+L)^2}
& = \sum\limits_{i=n+1}^{n+p}\frac{t}{(i+L)^2} \leq \sum\limits_{i=n+1}^{n+p}\frac{t}{(i+L-1)(i+L)} \\
& = \sum\limits_{i=n+1}^{n+p}\left(\frac{t}{i+L-1}-\frac{t}{i+L}\right) = \frac{t}{n+L}-\frac{t}{n+p+L} \\
& \leq  \frac{t}{n+L} \leq \frac1{k+1}. 
\end{align*}
\item The case $t=0$ is obvious. Assume that $t\ne 0$. Then for all $n\in\N$,
\begin{align*} 
\sum\limits_{i=0}^{n}\frac{t}{(i+L)^2} & = \frac{t}{L^2} + \sum\limits_{i=1}^n\frac{t}{(i+L)^2}
\leq  \frac{t}{L^2} + \frac{t}{L} \le \left\lceil t\left(\frac1L+\frac1{L^2}\right)\right\rceil = M_t.
\end{align*} 
As $\frac1L+\frac1{L^2}\leq 2$, we get that $M_t \leq 2\left\lceil t\right\rceil$.
\end{enumerate}
\end{proof}

\subsection{Example 1}

Assume that for all $n\in \N$, 
$$ \alpha_n=1-\lambda, \quad \beta_n=\lambda, \quad r_n=\frac1{(n+L)^2}r^*,$$
where $\lambda\in (0,1)$, $L\in \N^*$ and $r^*\in X$. 
Let us denote in the sequel 
$$\Lambda = \left\lceil\frac{1}{\lambda(1-\lambda)}\right\rceil. $$ 

\begin{proposition}\label{example1-result1}
$(x_n)$ is $T$-asymptotically regular with rate
\begin{align*}
\Phi^*(k) & = \Lambda\big(\defP^*(2k+1) + 8\left\lceil\|r^*\|\right\rceil(k+1)+1\big)
\end{align*}
 and asymptotically regular  with rate 
\begin{align*}
\Psi^*(k) & = \Lambda\left(\defP^*(4k+3) + 16\left\lceil\|r^*\|\right\rceil(k+1)+1\right),
\end{align*}
where $\defP^*$ is defined as in Corollary~\ref{main-thm-T-as-reg-cor-inexact-KM} with 
$\Mrn=2\left\lceil \|r^*\|\right\rceil$.
\end{proposition}
\begin{proof}
One can easily see that (C2*) holds  with $\sigma_2(k)=k\Lambda$. 
Furthermore, by applying Lemma \ref{Cauchy-modulus-bound-u-i+L2} with $t:=\|r^*\|$ we get that 
(C3) holds with $\sigma_3(k)=\left\lceil\|r^*\|\right\rceil(k+1)$
and that one can take $\Mrn$ as above. Apply  now Corollary \ref{main-thm-T-as-reg-cor-inexact-KM}. 
\end{proof}

In the setting of Hilbert spaces, we get quadratic rates of ($T$-)asymptotic regularity of $(x_n)$.

\begin{proposition}
Assume furthermore that $X$ is a Hilbert space. Then 
$(x_n)$ is $T$-asymptotically regular with rate
\begin{align*}
\Phi^*(k) & = 16\Lambda(\bxz + 2\left\lceil\|r^*\|\right\rceil)(\bxz + 4\left\lceil\|r^*\|\right\rceil + 1)(k+1)^2 + 8\Lambda\left\lceil\|r^*\|\right\rceil(k+1)+\Lambda
\end{align*}
 and asymptotically regular  with rate 
\begin{align*}
\Psi^*(k) & = 64\Lambda(\bxz + 2\left\lceil\|r^*\|\right\rceil)(\bxz + 4\left\lceil\|r^*\|\right\rceil + 1)(k+1)^2 + 16	\Lambda\left\lceil\|r^*\|\right\rceil(k+1) + \Lambda.
\end{align*}
\end{proposition}
\begin{proof}
Apply Proposition~\ref{example1-result1} with $\defP^*(k)$ defined by \eqref{inexact-KM-Hilbert-defP}.
\end{proof}

If, furthermore, $r^*=0$, then $(x_n)$ is the Krasnoselskii-Mann 
iteration and
$$ \Phi^*(k) = 16\Lambda\bxz(\bxz +1)(k+1)^2 +\Lambda \quad \text{and}\quad  \Psi^*(k) =64\Lambda\bxz  (\bxz + 1)(k+1)^2 + \Lambda. $$

The rate $\Phi^*$ of $T$-asymptotic regularity  is slightly worse (roughly, by a constant factor)
than the ones obtained for the Krasnoselskii-Mann iteration by Kohlenbach \cite[Corollary 1]{Koh01a} in Hilbert spaces
and by the second author \cite[Corollary 20]{Leu07} in CAT(0) spaces.

\subsection{Example 2}

Assume that for all $n\in \N$,
$$\alpha_n=\lambda, \quad \beta_n=1-\lambda-\frac1{(n+J)^2}, \quad r_n=\frac1{(n+L)^2}{r^*},$$
where $J\in \N, J\geq 2$, $L\in \N^*$, $\lambda \in \left(0,\frac{J^2-1}{J^2}\right)$ and $r^*\in X$. \\
Remark that for all $n\in \N$, $\alpha_n+\beta_n=1-\frac1{(n+J)^2}<1$ and $\beta_n\geq \beta_0=\frac{J^2-1}{J^2} - \lambda > 0$.

\begin{proposition}\label{rates-ex2}
$(x_n)$ is $T$-asymptotically regular with rate
\begin{align*}
\Phi(k) & = \Lambda\defP(2k+1) + 16\Lambda(2\bxz + \left\lceil \|r^*\|\right\rceil)(k+1) + 3\Lambda - 1
\end{align*}
and asymptotically regular with rate
\begin{align*}
\Psi(k) & = \Lambda\defP(4k+3) + 32\Lambda(2\bxz + \left\lceil \|r^*\|\right\rceil)(k+1) + 3\Lambda - 1,
\end{align*}
where $\displaystyle \Lambda = \left\lceil\frac{1}{\lambda(1-\lambda)}\right\rceil$. 
\end{proposition}
\begin{proof}
We can apply Lemma \ref{Cauchy-modulus-bound-u-i+L2} to get that
(C1) holds with $\sigma_1(k)=k+1$, (C3) holds with $\sigma_3(k)=\left\lceil\|r^*\|\right\rceil(k+1)$
and that one can take $\Mabn=2$, $\Mrn=2\left\lceil \|r^*\|\right\rceil$. \\

\noindent \textbf{Claim:} 
(C2) holds with rate of divergence $\sigma_2(n)=\Lambda(n+2)-1$.\\
\textbf{Proof of claim:}
We have that for all $n\in \N$, 
\begin{align*}
\sum_{i=0}^{\sigma_2(n)}\frac{\alpha_i\beta_i}{\alpha_i+\beta_i} & \geq  
\sum_{i=0}^{\sigma_2(n)}\alpha_i\beta_i \quad \text{as $\alpha_i+\beta_i = 1 - \frac1{(i+J)^2}<1$}\\
& = \sum_{i=0}^{\sigma_2(n)} \lambda\left(1-\lambda-\frac1{(i+J)^2}\right) 
 = (\sigma_2(n)+1) \lambda(1-\lambda) - \lambda\sum_{i=0}^{\sigma_2(n)} \frac1{(i+J)^2}  \\
& \geq  (\sigma_2(n)+1) \lambda(1-\lambda) - 2\lambda, \quad 
\text{by Lemma \ref{Cauchy-modulus-bound-u-i+L2}.\eqref{Cauchy-modulus-bound-u-i+L2-bound} with $t:=1$ and $L:=J$}\\
& = \Lambda(n+2) \lambda(1-\lambda) - 2\lambda \geq (n+2) - 2\lambda > n \quad \text{as $\lambda<1$.} \qquad  \qquad \qquad \qquad \qquad  \hfill \blacksquare
\end{align*}

By \eqref{def-bxz-bz-Mxnz}, we have that 
$$
\Mxnz = 3\bxz + 2\left\lceil \|r^*\|\right\rceil, \quad  
\Mxn = 4\bxz + 2\left\lceil \|r^*\|\right\rceil. 
$$
Apply now Theorem \ref{main-thm-T-as-reg} to conclude the proof.
\end{proof}

As an application of (the proof of) Proposition \ref{rates-ex2} and Corollary \ref{cor-Hilbert-spaces} we get 
quadratic rates of ($T$-)asymptotic regularity in the case of Hilbert spaces.

\begin{proposition}
Assume furthermore that $X$ is a Hilbert space. Then 
$(x_n)$ is $T$-asymptotically regular with rate
\begin{align*}
\Phi(k) & = 16\Lambda M_1(k+1)^2 + 16\Lambda M_2(k+1) + 3\Lambda - 1
\end{align*}
 and asymptotically regular  with rate 
\begin{align*}
\Psi(k) & = 64\Lambda M_1(k+1)^2 + 32\Lambda M_2(k+1) + 3\Lambda - 1, 
\end{align*}
where $M_1=(3\bxz + 2\left\lceil \|r^*\|\right\rceil)(5\bxz + 4\left\lceil \|r^*\|\right\rceil+ 1)$ 
and $M_2=2\bxz + \left\lceil \|r^*\|\right\rceil$.
\end{proposition}

\section{Proof of the main results}\label{main-results-proofs}

\begin{lemma}\label{lemma-M-upper-bound-xnz-xn} 
For all $n\in\N$, $\|x_n-z\|  \leq \Mxnz$ and $\|x_n\| \leq \Mxn$.
\end{lemma}
\begin{proof}
Apply Lemma \ref{lemma-upper-bound-xnz} and \eqref{def-bxz-bz-Mxnz} to get that 
\begin{align*}
\|x_n-z\| & \leq \|x-z\| +  \Mabn\|z\| + \Mrn \leq \bxz +  \Mabn \bxz + \Mrn = \Mxnz.
\end{align*}
Furthermore,
$\|x_n\|\leq \|x_n-z\|+\|z\|\leq \Mxnz + \bxz = \Mxn$.
\end{proof}

\subsection{A modulus of liminf}\label{section-modulus-liminf}

\begin{proposition}\label{xn-modulus-liminf}
$\liminf\limits_{n\to\infty} \|x_n-Tx_n\|=0$ with modulus of liminf $\Delta:\N\times\N\to\N$ defined by
\begin{equation}\label{def-modulus-liminf}
\Delta(k,L)=\sigma_2(\defP(k)+L),
\end{equation}
where $\defP$ is defined by \eqref{def-Pk}.
\end{proposition}
\begin{proof}
Assume by contradiction that $\Delta$ is not a modulus of liminf for $(\|x_n-Tx_n\|)$. 
Then there are $k, L\in\N$ such that 
\begin{center}
 $\|x_n-Tx_n\|\geq \frac1{k+1}$ for all $n\in [L, \Delta(k,L)]$.
\end{center}
Let $n\in [L, \Delta(k,L)]$ be arbitrary. Then $x_n \ne Tx_n$, so $x_n \ne z$, that is 
$\|x_n-z\|>0$. By (C2) and Lemma \ref{lemma-M-upper-bound-xnz-xn}, we have that 
$\alpha_n+\beta_n >0$ and $\|x_n-z\| \leq \Mxnz$. 
Thus, we can apply Lemma~\ref{main-lemma-xn-uc}.\eqref{main-lemma-xn-uc-basic}
with $\delta=\Mxnz$, $\gamma=\|x_n-z\|$ and $\theta =\frac1{k+1}$ to get that 
\begin{align*}
 \|\alpha_nx_n+\beta_nTx_n-(\alpha_n+\beta_n)z\| & \leq \|x_n-z\| -
\frac{2\|x_n-z\|\alpha_n\beta_n}{\alpha_n+\beta_n}\eta\left(\frac{1}{\Mxnz(k+1)}\right).
\end{align*}
Since, by Corollary \ref{cor-xn-Txn-xn-z-fixed-point},  $2\|x_n-z\| \ge \|x_n-Tx_n\| \ge \frac{1}{k+1}$, 
we get that 
\begin{align}\label{mod-inf-prelim-1}
 \|\alpha_nx_n+\beta_nTx_n-(\alpha_n+\beta_n)z\| & \leq \|x_n-z\| 
 -\frac{\alpha_n\beta_n}{(\alpha_n+\beta_n)(k+1)}\eta\left(\frac{1}{\Mxnz(k+1)}\right).
\end{align}
As
\begin{align*}
\|x_{n+1}-z\| &= \|\alpha_nx_n+\beta_nTx_n + r_n - z\|\\
&= \|(\alpha_nx_n+\beta_nTx_n-(\alpha_n+\beta_n)z) - (1- \alpha_n - \beta_n)z + r_n\|\\
&\le \|\alpha_nx_n+\beta_nTx_n-(\alpha_n+\beta_n)z\| + (1-\alpha_n - \beta_n)\|z\| + \|r_n\|,
\end{align*}
it follows from \eqref{mod-inf-prelim-1} that 
\begin{align*}
\|x_{n+1}-z\| & \leq \|x_n-z\|-\frac{\alpha_n\beta_n}{(\alpha_n+\beta_n)(k+1)}\eta\left(\frac{1}{\Mxnz(k+1)}\right)
+(1-\alpha_n-\beta_n)\|z\| + \|r_n\|.
\end{align*}
Adding up the above for $n=L,\ldots,\Delta(k,L)$, we get that 
\begin{align*}
\|x_{\Delta(k,L)+1} - z\| & \le \|x_L-z\|- \frac{1}{k+1}\eta\left(\frac{1}{\Mxnz(k+1)}\right)
\sum\limits_{n=L}^{\Delta(k,L)}\frac{\alpha_n\beta_n}{\alpha_n+\beta_n} \\
& + \|z\|\sum\limits_{n=L}^{\Delta(k,L)}(1-\alpha_n-\beta_n) +  
\sum\limits_{n=L}^{\Delta(k,L)}\|r_n\| \\
& \stackrel{\eqref{Mabn-Mrn-upper-bounds-series}}{\le}  \|x_L-z\|- \frac{1}{k+1}\eta\left(\frac{1}{\Mxnz(k+1)}\right)
\sum\limits_{n=L}^{\Delta(k,L)}\frac{\alpha_n\beta_n}{\alpha_n+\beta_n} + 
\|z\|\Mabn + \Mrn \\
& \le \Mxnz  - \frac{1}{k+1}\eta\left(\frac{1}{\Mxnz(k+1)}\right)
\sum\limits_{n=L}^{\Delta(k,L)}\frac{\alpha_n\beta_n}{\alpha_n+\beta_n} + 
\|z\|\Mabn + \Mrn \quad \text{by Lemma \ref{lemma-M-upper-bound-xnz-xn}}\\
& \stackrel{\eqref{def-bxz-bz-Mxnz}}{\le} (\Mxnz + \Mabn\bxz  + \Mrn) - \frac{1}{k+1}\eta\left(\frac{1}{\Mxnz(k+1)}\right)
\sum\limits_{n=L}^{\Delta(k,L)}\frac{\alpha_n\beta_n}{\alpha_n+\beta_n}.
\end{align*}
Remark that 
\begin{align*}
\sum\limits_{n=L}^{\Delta(k,L)}\frac{\alpha_n\beta_n}{\alpha_n+\beta_n} 
& = \sum\limits_{n=L}^{\sigma_2(\defP(k)+L)}\frac{\alpha_n\beta_n}{\alpha_n+\beta_n} 
 = \sum\limits_{n=0}^{\sigma_2(\defP(k)+L)}\frac{\alpha_n\beta_n}{\alpha_n+\beta_n} 
 - \sum\limits_{n=0}^{L-1}\frac{\alpha_n\beta_n}{\alpha_n+\beta_n} \\
 &  \geq (\defP(k)+L) - \sum\limits_{n=0}^{L-1}\frac{\alpha_n\beta_n}{\alpha_n+\beta_n} \quad \text{by (C2)}\\
& \ge (\defP(k)+L)-L = \defP(k), \quad 
\text{as~}\frac{\alpha_n\beta_n}{\alpha_n+\beta_n}
\leq \alpha_n \leq 1.
\end{align*}
It follows that 
\begin{align*}
\|x_{\Delta(k,L)+1} - z\| &  \le (\Mxnz + \Mabn\bxz  + \Mrn) -  
\frac{\defP(k)}{k+1}\eta\left(\frac{1}{\Mxnz(k+1)}\right)\\ 
& \stackrel{\eqref{def-Pk}}{=} (\Mxnz + \Mabn\bxz  + \Mrn) - 
\left\lceil\frac{ (\Mxnz + \Mabn\bxz  + \Mrn + 1)(k+1)}{\eta\left(\frac{1}{\Mxnz(k+1)}\right)}\right\rceil
\frac{\eta\left(\frac{1}{\Mxnz(k+1)}\right)}{k+1} \\
& \leq (\Mxnz + \Mabn\bxz  + \Mrn) - (\Mxnz + \Mabn\bxz  + \Mrn + 1)=-1,
\end{align*}
which is a contradiction.  Thus, $\Delta$ is a modulus of liminf for $(\|x_n-Tx_n\|)$.
\end{proof}

\begin{proposition}\label{xn-modulus-liminf-eta-special}
Assume that $\eta$ can be written as $\eta(\varepsilon)=\varepsilon\cdot\tilde{\eta}(\varepsilon)$ 
with $\tilde{\eta}$ increasing. 

Then $\liminf\limits_{n\to\infty} \|x_n-Tx_n\|=0$ with modulus of liminf 
$\tilde{\Delta}:\N\times\N\to\N$ defined by
\begin{equation}\label{def-modulus-liminf-eta-special}
\tilde{\Delta}(k,L)=\sigma_2(\deftP(k)+L),
\end{equation}
where $\deftP$ is defined by \eqref{def-Ptk-nice-modulus}.
\end{proposition}
\begin{proof}
Follow the proof of Proposition~\ref{xn-modulus-liminf}, but apply
Lemma~\ref{main-lemma-xn-uc}.\eqref{main-lemma-xn-uc-special} instead of 
Lemma~\ref{main-lemma-xn-uc}.\eqref{main-lemma-xn-uc-basic}.
\end{proof}

\subsection{Proof of Theorem~\ref{main-thm-T-as-reg} and Proposition~\ref{main-thm-T-as-reg-special-eta}}

By \eqref{ineq-xnp1-Txnp1} and Lemma \ref{lemma-M-upper-bound-xnz-xn} we get that for all $n\in\N$, 
\begin{align*}
\|x_{n+1}-Tx_{n+1}\| & \le \|x_n-Tx_n\| + 2(1-\alpha_n-\beta_n)\|x_n\| + 2\|r_n\| \\
& \le \|x_n-Tx_n\| + 2\Mxn(1-\alpha_n-\beta_n) + 2\|r_n\|.
\end{align*}
Let us verify that the hypotheses of Proposition~\ref{prop-rate-conv-modulus-liminf} hold  
with 
\begin{center} 
$a_n:=\|x_n-Tx_n\| $ and $b_n:= 2\Mxn(1-\alpha_n-\beta_n)  + 2\|r_n\|$.
\end{center}
By Proposition~\ref{xn-modulus-liminf}, $\liminf\limits_{n\to\infty} a_n=0$ with modulus 
of liminf $\Delta$ given by \eqref{def-modulus-liminf}. Furthermore, by (C1), (C3) and 
Lemma~\ref{cauchy-modulus-linear-comb}, we get that 
$\psi(k)=\max\{\sigma_1(4\Mxn(k+1)-1), \sigma_3(4k+3)\}$ is a Cauchy modulus of $(b_n)$. 

Thus, we can apply Proposition~\ref{prop-rate-conv-modulus-liminf} to get that  
$\lim\limits_{n\to\infty} \|x_n-Tx_n\| =0$ with rate of convergence
\begin{align*}
\Phi(k) & = \Delta\left(2k+1, \psi(2k+1)+1\right) =
\sigma_2(\defP(2k+1)+\psi(2k+1)+1)\\
& = \sigma_2\left(\defP(2k+1)+\max\{\sigma_1(4\Mxn(2k+2)-1), \sigma_3(4(2k+1)+3)\}+1\right)\\
& = \sigma_2\big(\defP(2k+1)+\max\{\sigma_1(8\Mxn(k+1)-1), \sigma_3(8k+7)\}+1\big).
\end{align*}

Let us prove now that $\Psi$ defined by \eqref{def-Psi-rate-as-reg} is a rate of asymptotic 
regularity. First, one can easily see  that $\frac{\alpha_n\beta_n}{\alpha_n+\beta_n} < 1$ for every $n\in\N$. 
Thus, by Lemma \ref{divergence-rate-geq-k}, $\sigma_2(k) \ge k$. 
We get that 
\begin{align*}
\Psi(k)&=\Phi(2k+1)=\sigma_2\left(\defP(4k+3)+\max\{\sigma_1(16\Mxn(k+1)-1), \sigma_3(16k+15)\}+1\right)\\
&\ge \defP(4k+3)+\max\{\sigma_1(16\Mxn(k+1)-1), \sigma_3(16k+15)\}+1\\
&\ge \max\{\sigma_1(16\Mxn(k+1)-1), \sigma_3(16k+15)\}+1.
\end{align*}
Let $n\geq \Psi(k)$. It follows that 
\begin{align*}
\|x_{n+1}-x_n\| & \stackrel{\eqref{ineq-xnp1-xn}}{\le} \beta_n\|Tx_n-x_n\|+(1-\alpha_n-\beta_n)\|x_n\| + \|r_n\| \\
& \le \|Tx_n-x_n\|+(1-\alpha_n-\beta_n)\Mxn + \|r_n\| \quad \text{by Lemma \ref{lemma-M-upper-bound-xnz-xn} }\\
& \le  \frac1{2(k+1)} + (1-\alpha_n-\beta_n)\Mxn + \|r_n\| \quad \text{as $n\geq \Phi(2k+1)$}\\
& \le  \frac1{2(k+1)} + (1-\alpha_n-\beta_n)\Mxn + \frac1{16(k+1)} \\
& \text{by Lemma \ref{lemma-series-conv-upper-bound}.\eqref{lemma-series-conv} and the fact that 
$n\geq \sigma_3(16k+15)+1$}\\
& \le  \frac1{2(k+1)} +  \frac1{16(k+1)} +  \frac1{16(k+1)} < \frac1{k+1} \\
& \text{by Lemma \ref{lemma-series-conv-upper-bound}.\eqref{lemma-series-conv} and the fact that 
$n\geq \sigma_1(16\Mxn(k+1)-1)+1$.}
\end{align*}
Thus, Theorem~\ref{main-thm-T-as-reg} is proved.  

The proof of Proposition~\ref{main-thm-T-as-reg-special-eta} follows 
the same line, with the difference that we apply Proposition~\ref{xn-modulus-liminf-eta-special} 
instead of Proposition~\ref{xn-modulus-liminf}.

\mbox{}

\noindent \textbf{Acknowledgement} The first author acknowledges the support of FCT – Fundação
para a Ciência e Tecnologia through a doctoral scholarship with reference number
2022.12585.BD as well as the support of the research center CEMS.UL under the
FCT funding \href{https://doi.org/10.54499/UIDB/04561/2020}{UIDB/04561/2020}.


\begin{thebibliography}{10}

\bibitem{BorReiSha92}
J.~Borwein, S.~Reich, and I.~Shafrir.
\newblock {Krasnoselski-Mann} iterations in normed spaces.
\newblock {\em Canadian Mathematical Bulletin}, 35:21--28, 1992.

\bibitem{BroPet66}
F.~Browder and W.~Petryshyn.
\newblock The solution by iteration of nonlinear functional equations in
  {B}anach spaces.
\newblock {\em Bulletin of the American Mathematical Society}, 72:571--575,
  1966.

\bibitem{Com04}
P.~L. Combettes.
\newblock Solving monotone inclusions via compositions of nonexpansive averaged
  operators.
\newblock {\em Optimization}, 53(5-6):475--504, 2004.

\bibitem{FirLeu24}
P.~Firmino and L.~Leu\c{s}tean.
\newblock {Quantitative asymptotic regularity of the VAM iteration with error
  terms for accretive operators in Banach spaces}.
\newblock {\em Zeitschrift f\"{u}r Analysis und ihre Anwendugen}, DOI
  10.4171/ZAA/1772, 2024.

\bibitem{Gro72}
C.~Groetsch.
\newblock A note on segmenting {M}ann iterates.
\newblock {\em Journal of Mathematical Analysis and Applications}, 40:369--372,
  1972.

\bibitem{Hu08}
L.-G. Hu.
\newblock {Strong convergence of a modified Halpern's iteration for
  nonexpansive mappings}.
\newblock {\em Fixed Point Theory and Applications}, page Article ID 649162,
  2008.

\bibitem{HuLiu09}
L.-G. Hu and L.~Liu.
\newblock A new iterative algorithm for common solutions of a finite family of
  accretive operators.
\newblock {\em Nonlinear Analysis}, 70:2344--2351, 2009.

\bibitem{KanShe17}
C.~Kanzow and Y.~Shehu.
\newblock {Generalized Krasnoselskii--Mann-type iterations for nonexpansive
  mappings in Hilbert spaces}.
\newblock {\em Computational Optimization and Applications}, 67:595--620, 2017.

\bibitem{KimXu07}
T.-H. Kim and H.-K. Xu.
\newblock {Robustness of Mann's algorithm for nonexpansive mappings}.
\newblock {\em Journal of Mathematical Analysis and Applications},
  327:1105--1115, 2007.

\bibitem{Koh01a}
U.~Kohlenbach.
\newblock On the computational content of the krasnoselski and ishikawa fixed
  point theorems.
\newblock In J.~Blanck, V.~Brattka, P.~Hertling, and W.~K., editors, {\em
  Proceedings of the Fourth Workshop on Computability and Complexity in
  Analysis}, volume~15 of {\em Lecture Notes in Computer Science}, pages
  119--145. Springer, 2001.

\bibitem{Koh03}
U.~Kohlenbach.
\newblock Uniform asymptotic regularity for {M}ann iterates.
\newblock {\em Journal of Mathematical Analysis and Applications},
  279:531--544, 2003.

\bibitem{Koh08}
U.~Kohlenbach.
\newblock {\em {Applied Proof Theory: Proof Interpretations and their Use in
  Mathematics}}.
\newblock Springer, 2008.

\bibitem{Koh19}
U.~Kohlenbach.
\newblock Proof-theoretic methods in nonlinear analysis.
\newblock In B.~Sirakov, P.~Ney~de Souza, and M.~Viana, editors, {\em
  Proceedings of ICM 2018, Vol. 2}, pages 61--82. World Scientific, 2019.

\bibitem{KohLeu12}
U.~Kohlenbach and L.~Leu\c{s}tean.
\newblock On the computational content of convergence proofs via {B}anach
  limits.
\newblock {\em Philsophical Transactions of the Royal Society Series A:
  Mathematical, Physical and Engineering Sciences}, 370:3449--3463, 2012.

\bibitem{Kra55}
M.~A. Krasnosel'skii.
\newblock Two remarks on the method of successive approximation.
\newblock {\em Uspekhi Mathematicheskikh Nauk}, 10:123--127, 1955.

\bibitem{Leu07}
L.~Leu\c{s}tean.
\newblock A quadratic rate of asymptotic regularity for {CAT(0)}-spaces.
\newblock {\em Journal of Mathematical Analysis and Applications},
  325:386--399, 2007.

\bibitem{Leu10}
L.~Leu\c{s}tean.
\newblock Nonexpansive iterations in uniformly convex {$W$}-hyperbolic spaces.
\newblock In B.~S. Mordukhovich, I.~Shafrir, and A.~Zaslavski, editors, {\em
  Nonlinear Analysis and Optimization I: Nonlinear Analysis}, volume 513 of
  {\em Contemporary Mathematics}, pages 193--209. American Mathematical
  Society, 2010.

\bibitem{Leu14}
L.~Leu\c{s}tean.
\newblock An application of proof mining to nonlinear iterations.
\newblock {\em Annals of Pure and Applied Logic}, 165:1484--1500, 2014.

\bibitem{LiaFadPey16}
J.~Liang, J.~Fadili, and G.~Peyr{\'e}.
\newblock Convergence rates with inexact non-expansive operators.
\newblock {\em Mathematical Programming, Ser. A}, 159:403--434, 2016.

\bibitem{LinTza79}
J.~Lindenstrauss and L.~Tzafriri.
\newblock {\em Classical {Banach} spaces II: {Function} spaces}.
\newblock Springer, Berlin-Heidelberg, 1979.

\bibitem{Man53}
W.~R. Mann.
\newblock Mean value methods in iteration.
\newblock {\em Proceedings of the American Mathematical Society}, 4:506--510,
  1953.

\bibitem{YaoLioZho09}
Y.~Yao, Y.-C. Liou, and H.~Zhou.
\newblock Strong convergence of an iterative method for nonexpansive mappings
  with new control conditions.
\newblock {\em Nonlinear Analysis}, 70:2332--2336, 2009.

\bibitem{ZhaGuoWan21}
Y.-C. Zhang, K.~Guo, and T.~Wang.
\newblock {Generalized Krasnoselskii-Mann-Type iteration for nonexpansive
  mappings in Banach spaces}.
\newblock {\em Journal of the Operations Research Society of China},
  9:195--206, 2021.

\end{thebibliography}
\end{document}